\documentclass[12pt]{amsart}
\usepackage{amstext,amsfonts,amssymb,amscd,amsbsy,amsmath,verbatim, mathrsfs, fullpage}
\usepackage[alphabetic,abbrev,lite]{amsrefs} 
\usepackage{ifthen,tikz}
\usepackage{color}
\usepackage{amsthm}
\usepackage{amsmath}
\usepackage{latexsym}
\usepackage[all]{xy}
\usepackage{enumerate}
\usepackage{mathtools}

\numberwithin{equation}{section}
\newtheorem{lemma}[equation]{Lemma}

\newtheorem{prop}[equation]{Proposition}
\newtheorem{cor}[equation]{Corollary}

\newtheorem{claim*}{Claim}
\newtheorem{thm}[equation]{Theorem}

\theoremstyle{definition}
\newtheorem{defn}[equation]{Definition}
\newtheorem{notation}[equation]{Notation}
\newtheorem{example}[equation]{Example}
\newtheorem{construction}[equation]{Construction}

\theoremstyle{remark}
\newtheorem{remark}[equation]{Remark}

\newcommand{\cC}{\mathcal{C}}

\newcommand{\PP}{\mathbb P}

\renewcommand{\P}{\PP}

\newcommand{\ZZ}{\mathbb Z}

\newcommand{\Spec}{\operatorname{Spec}}

\newcommand{\Eff}{\operatorname{Eff}}
\newcommand{\im}{\operatorname{im}}

\newcommand{\id}{\operatorname{id}}

\newcommand{\Tor}{\operatorname{Tor}}

\newcommand{\Hom}{\operatorname{Hom}} 

\newcommand{\kk}{{\bf k}}

\newcommand{\F}{\FF}

\newcommand{\defi}[1]{\textsf{#1}} 

\newcommand{\beq}{\begin{displaymath}}
\newcommand{\eeq}{\end{displaymath}}

\def\nc{\newcommand}
\def\on{\operatorname}
\nc{\Q}{\mathbb{Q}}
\nc{\RR}{\mathbf{R}}
\nc{\LL}{\mathbf{L}}
\nc{\xra}{\xrightarrow}
\nc{\xla}{\xleftarrow}
\def\a{\alpha}
\def\om{\omega}

\def\DM{\operatorname{DM}}

\def\th{\on{th}}
\def\F{\mathcal{F}}
\def\coker{\on{coker}}

\nc{\into}{\hookrightarrow}
\nc{\onto}{\twoheadrightarrow}
\nc{\OO}{\mathcal{O}}
\nc{\Z}{\mathbb{Z}}
\nc{\cA}{\mathcal{A}}
\nc{\w}{\widehat}
\nc{\End}{\on{End}}
\nc{\res}{\frac{1}{x_0x_1}}
\nc{\tF}{\widetilde{F}}
\nc{\tG}{\widetilde{G}}
\nc{\tf}{\widetilde{f}}
\nc{\Com}{\on{Com}}

\nc{\zero}{\mathbf{0}}
\nc{\G}{\mathbb{G}}
\nc{\cG}{\mathcal{G}}
\nc{\cE}{\mathcal E}
\nc{\cF}{\mathcal F}
\nc{\cR}{\mathcal R}
\nc{\cD}{\mathcal D}
\nc{\cB}{\mathcal B}
\nc{\cT}{\mathcal T}
\nc{\cL}{\mathcal L}

\nc{\bM}{\mathbf M}
\nc{\bN}{\mathbf N}
\nc{\U}{\mathbf U}
\nc{\BM}{\mathbf B \mathbf M}
\nc{\Dsg}{\on{D}_{\on{sg}}}
\nc{\fC}{\mathcal{C}}
\nc{\fG}{\mathcal{G}}
\nc{\N}{\mathbb{N}}

\nc{\del}{\partial}
\nc{\cone}{\on{cone}}
\nc{\D}{\on{D}_{\on{diff}}}
\nc{\DMb}{\on{D}^b_{\DM}}
\nc{\Db}{\on{D}^{\on{b}}}
\nc{\Kb}{\on{K}^{\on{b}}}
\nc{\fm}{\mathfrak{m}}
\nc{\Flag}{\on{Flag}}
\nc{\DMmin}{\DM_{\on{min}}}
\nc{\Ddiff}{\on{D}_{\on{diff}}}
\nc{\Dbdiff}{\on{D}^\on{b}_{\on{diff}}}
\nc{\wO}{\widehat{\OO}}
\nc{\wT}{\widehat{T}}
\nc{\from}{\leftarrow}
\nc{\wLL}{\widetilde{\LL}}
\nc{\augCech}{\widetilde{\cC}}
\nc{\Fold}{\on{Fold}}
\nc{\Ext}{\on{Ext}}
\nc{\FF}{\mathbf{F}}
\nc{\Comper}{\Com_{\on{per}}}
\nc{\Unfold}{\on{Unfold}}
\nc{\intHom}{\underline{\Hom}}
\nc{\Ex}{\on{Ex}}
\nc{\tg}{\widetilde{g}}

\nc{\B}{\mathcal{B}}
\nc{\K}{\mathcal{K}}
\nc{\kos}{\on{Kos}}
\nc{\Perf}{\on{Perf}}
\nc{\tR}{\widetilde{\cR}}
\nc{\X}{\mathcal{X}}
\nc{\Cl}{\on{Cl}}
\nc{\fU}{\mathcal{U}}
\nc{\bU}{\mathbf U}

\nc{\st}{\on{st}}

\nc{\coh}{\on{coh}}

\def\D{\mathcal{D}}

\nc{\tU}{\U}
\nc{\bC}{\mathbf{C}}
\nc{\aux}{\on{aux}}

\def\phi{\varphi}
\newcommand{\Manoa}{M\=anoa}
\newcommand{\Hawaii}{Hawai\kern.05em`\kern.05em\relax i}
\newcommand{\UHM}{University of \Hawaii \ at \Manoa}

\title{Linear strands of multigraded free resolutions}
\author{Michael K. Brown}
\address{Department of Mathematics, Auburn University, Auburn, AL}
\email{mkb0096@auburn.edu}

\author{Daniel Erman}
\address{Department of Mathematics, \UHM, Honolulu, HI}
\email{erman@hawaii.edu}
\date{\today}

\begin{document}

\maketitle


\begin{abstract}
We develop a notion of linear strands for multigraded free resolutions, and we prove a multigraded generalization of Green's Linear Syzygy Theorem. 
\end{abstract}
\section{Introduction}
Linear strands of minimal free resolutions over a standard graded polynomial ring play an important role in commutative algebra and projective geometry; see e.g. \cite{EPS, GL, mccullough, RW, SS}, and see \cite[Chapter 7]{eisenbud} for a comprehensive introduction to the subject. The main goal of this paper is to develop a notion of linear strands of resolutions over \defi{multigraded} polynomial rings, by which we mean polynomial rings graded by an arbitrary abelian group.  Our paper parallels other recent homological results in multigraded commutative algebra~\cite{bruce-heller-sayrafi,EES, MS, tate} in the following sense: while the final results are strong analogues of results from the standard graded case, including an analogue of the main result of~\cite{green}, getting to those results requires reconceiving the central definitions and developing novel proof techniques.

Multigraded polynomial rings arise, for instance, as the coordinate rings of toric varieties/stacks~\cite{cox}. One particularly interesting class of examples is given by nonstandard $\ZZ$-graded polynomial rings, which correspond to weighted projective {spaces/stacks}, and our work on linear strands is novel even in this setting.

 A highlight of the classical theory of linear strands is its connection with Green's ``$N_p$-properties" \cite{green2}, which measure the linearity of free resolutions of coordinate rings of curves embedded in $\mathbb P^n$.  This provides a major motivation for the current work, and we specifically aim to provide a foundation for our work on $N_p$-properties for curves in weighted projective spaces \cite{Np}.

This work also adds to the recent overarching program of developing homological tools and results for understanding syzygies in the multigraded setting.  This includes: work of Maclagan-Smith and others on Castelnuovo-Mumford regularity in the multigraded context~\cite{botbol-chardin, positivity, bruce-heller-sayrafi,MS,MS2}; the introduction of virtual resolutions as a framework for studying multigraded syzygies of varieties in~\cite{BES} and the followup results \cite{berkesch-klein-loper-yang,duarte-seceleanu, gllm, hnv, loper,    yang}; generalizations of the Bernstein-Gel'fand-Gel'fand correspondence \cite{baranovsky, bgg,  hhw}, which will be essential for the results in this paper; and extensions of Beilinson's resolution of the diagonal~\cite{beilinson, BS22, tate, CK } and Tate resolutions~\cite{tate, EES, EFS} to toric settings.

\subsection*{Results} To discuss our results in detail, we must fix some notation. Let $\kk$ be a field, and let $S = \kk[x_0, \dots, x_n]$, graded by an abelian group $A$. We denote by $\zero$ the identity in $A$. We will assume that $S$ has a \defi{positive $A$-grading}, by which we mean that $S_\zero = \kk$ and that there exists a homomorphism $\theta\colon A \to \Z$ such that $\theta(\deg(x_i)) > 0$ for all $i$.  Our main example is the case where $S$ is the Cox ring of a projective toric variety $X$ and $A=\Cl(X)$ is the divisor class group of $X$.

Let $M$ be a finitely generated graded $S$-module that is generated in degree $\mathbf{0}$ and $F$ its minimal $A$-graded free resolution (such resolutions always exist; see e.g. \cite[Lemma 3.11(1)]{botbol-chardin}). In the case where $S$ is standard graded, i.e. $A = \Z$ and $\deg(x_i) = 1$ for all $i$, the \defi{linear strand} of $F$ is the subcomplex generated by elements of each $F_i$ of degree $i$. In other words, the linear strand is the subcomplex given by the summands of $F$ that come from the first row of its Betti table;  it has the form $S^{b_0}\gets S(-1)^{b_1} \gets S(-2)^{b_2}\gets \cdots$, and the entries of the differentials are $\kk$-linear combinations of the variables.

To define the linear strand of a multigraded free resolution, we first ask: what does it mean for a multigraded complex to be ``linear"?  A key observation is that the degrees of the generators of the free modules may be insufficient to determine linearity.  As a simple example, take $S=\kk[x_0,x_1]$ with $\deg(x_0)=1$ and $\deg(x_1)=2$, and consider the complexes 
\begin{equation}\label{eqn:lin vs strong lin}
S \overset{\ x_1}{\longleftarrow} S(-2)\qquad \text{ and } \qquad S \overset{\ x_0^2}{\longleftarrow} S(-2) .
\end{equation}
The complexes are identical as modules, but intuitively, the one on the left is ``linear'' while the one on the right is not.  Thus, outside of the standard graded case, our notion of linearity must involve properties of the differentials, not just the underlying free modules.

\begin{defn}
\label{stronglylinear}
We call a morphism $f\colon G \to G'$ of free $A$-graded $S$-modules \defi{strongly linear} if there exist bases of $G$ and $G'$ with respect to which $f$ is a matrix whose entries are $\kk$-linear combinations of the variables. A complex $C$ of free $A$-graded $S$-modules is \defi{strongly linear} if its differentials are such.
\end{defn}

Since strong linearity cannot be detected using Betti numbers, our approach to defining multigraded linear strands is necessarily different from the standard graded case. To guide the way to our definition, we first consider a key property of linear strands in the standard graded setting:

\begin{quote}
When $S$ is standard graded, the linear strand of the free resolution $F$ is the (unique) maximal subcomplex $L$ of $F$ such that $L$ is linear (i.e. each term $L_i$ is generated in degree $i$), and $L$ is a summand of $F$ as an $S$-module.\footnote{This characterization relies on the fact that $M$ is generated entirely in degree $0$; see Remark~\ref{moregeneral}.} 
\end{quote}
\noindent From now on, we will say a subcomplex $C'$ of a complex $C$ of $A$-graded $S$-modules is \defi{quasi-split} if $C'$ is a summand of $C$ as an $S$-module (but not necessarily as a complex).

There is also a useful formula for computing linear strands via the Bernstein-Gel'fand-Gel'fand (BGG) correspondence \cite[Corollary 7.11]{eisenbud}. Our main result is that maximal linear subcomplexes of minimal free resolutions exist and are unique even in the multigraded setting, and that they can also be computed via a multigraded analogue of the BGG correspondence:
\begin{thm}[see Theorem \ref{maxlinear} below]
\label{introthm}
Let $\kk$ be a field, and suppose $S = \kk[x_0, \dots, x_n]$ is positively graded by an abelian group $A$. Let $M$ be a finitely generated graded $S$-module that is generated in degree $\mathbf{0}$ and $F$ its minimal free resolution.
\begin{enumerate}
\item There exists a unique maximal strongly linear, quasi-split subcomplex $L$ of $F$.
\item The complex  $L$ in (1) can be expressed in terms of the Bernstein-Gel'fand-Gel'fand (BGG) correspondence as $L=\LL(K)$, where
$$
K = \ker \left(M_{\zero} \otimes \om_E \xra{\sum_{i = 0}^n x_i \otimes e_i} \RR (M)\right)
$$
is a module over the Koszul dual exterior algebra $E$ of the ring $S$ (see \S\ref{subsec:bgg} for an explanation of the notation here).
\end{enumerate}
\end{thm}

We define the \defi{strongly linear strand} of the resolution $F$ to be the subcomplex $L$ as in Theorem \ref{introthm}.  We go on to define strongly linear strands in greater generality, without the assumption that $M$ is generated in degree $\zero$; see Definition \ref{moregeneral} below.
The formula in Theorem \ref{introthm}(2) can be used to efficiently compute strongly linear strands in \verb|Macaulay2|~\cite{M2}.

Theorem \ref{introthm}(2) is known in the standard graded case \cite[Corollary 7.11]{eisenbud}, but the arguments we use to prove Theorem \ref{introthm}(2) are completely different from those used to prove that result. The main difference is that the proof of \cite[Corollary 7.11]{eisenbud} uses the interpretation of linearity in terms of growth of Betti numbers, but this interpretation is unavailable outside of the standard graded case. Our approach is more delicate, involving a homological perturbation argument (Theorem \ref{samemods}) and an explicit calculation of the unit and counit of the multigraded BGG adjunction (Lemma \ref{technical}). 
Some aspects of these technical results have analogues in work of Eisenbud-Fl\o ystad-Schreyer (e.g. \cite[Theorem 3.7(b)]{EFS}), but our approach is quite different from theirs. The basic difficulty that arises is that the role of complexes over an exterior algebra in the classical BGG correspondence is played by differential modules over an exterior algebra in the multigraded case.  Thus, many of the techniques used in the proof of key technical results like \cite[Theorem 3.7(b)]{EFS} are not available in our setting; see Remark \ref{lemmaremark} for details.

\medskip

Finally, we turn to an application of Theorem~\ref{introthm}.  One of the most important results about linear strands in the standard graded case is Green's Linear Syzygy Theorem~\cite{green} (see \cite[Theorem 7.1]{eisenbud} for a modern statement), which gives a bound on the length of the linear strand in terms of simple invariants of $M$. Green's result answered positively a conjecture of Eisenbud-Koh~\cite{ek} and is the basis of perhaps the simplest modern proof of Green's seminal results on $N_p$ conditions for high degree curves (see \cite[Theorem 8.8(1)]{eisenbud}). We generalize Green's theorem to the multigraded setting:

\begin{thm}[Multigraded Linear Syzygy Theorem]
\label{intromlst}
Let $A$, $S$, $M$, and $F$ be as in Theorem~\ref{introthm}.  The length of the strongly linear strand of $F$ is at most $\max\{\dim M_\zero - 1, \dim R_\zero(M)\}$,
where $R_\zero(M)$ is the variety of rank one linear syzygies of $M$.
\end{thm}

We refer the reader to Notation \ref{notation} below for a precise definition of $R_\zero(M)$. See Theorem~\ref{LST} for a more general statement involving modules that are not necessarily generated in a single degree. Both of the bounds $\dim M_\zero - 1$ and $\dim R_\zero(M)$ in the theorem may be attained: see Remark~\ref{rem:sharp}.  The proof of Theorem~\ref{intromlst}  combines Green's proof of the Linear Syzygy Theorem and our theory of multigraded linear strands, as articulated in Theorem~\ref{introthm}.  In other words, the main novelty in Theorem~\ref{intromlst} is in the statement, and the main subtleties are in the development of a theory of multigraded linear strands, as discussed above.

As in Green's work, the Multigraded Linear Syzygy Theorem has implications for the geometric study of syzygies, namely, for the syzygies of vector bundles on subvarieties of a toric variety. Before stating our main result along these lines, we note: for a homogeneous ideal $I$ in a multigraded polynomial ring $S$, we say that $I$ is \defi{nondegenerate} if $I$ belongs to the square of the maximal ideal of $S$.

\begin{cor}
\label{greencor variety}
Let $X$ be a projective toric variety with Cox ring $S$ and $P\subseteq S$ a nondegenerate, homogeneous prime ideal with $Y \subseteq X$ the corresponding integral subvariety.  Let $N$ be a finitely generated graded $S$-module such that the sheaf $\widetilde{N}$ is a vector bundle on $Y$, and let $M$ be a submodule of the $A = \Cl(X)$-graded $S$-module $\bigoplus_{d \in \Eff(S)} H^0 (Y, \widetilde{N(d)})$. 
\begin{enumerate}
\item $\Eff(M) = \{a \in A \text{ : } M_a \ne 0\}$ contains a minimal element with respect to the partial ordering described in Notation \ref{effective}.
\item Let $a \in \Eff(M)$ be a minimal element. The $a$-strongly linear strand (see Definition \ref{moregeneral}) of the minimal free resolution of $M$ has length at most $\dim_\kk(M_a) - 1$. 
\end{enumerate}
\end{cor}
A nearly identical result holds when $X$ is a projective toric stack: see Corollary~\ref{greencor}.
See also \cite[Theorem 4.5]{Np} for a similar geometric consequence of the Multigraded Linear Syzygy Theorem that plays a crucial role in the proofs of the main results in \cite{Np}.

\subsection*{Funding}
The first author was supported by NSF-RTG grant 1502553.  The second author was supported by NSF grants 
DMS-1601619 and DMS-1902123.

\subsection*{Acknowledgements}
We thank David Eisenbud, Hal Schenck, and Frank-Olaf Schreyer for valuable conversations. We are also grateful to the referee for helpful comments that improved this paper. 

\subsection*{Notation}

Throughout the paper, $\kk$ denotes a field, and $A$ is an abelian group. We will assume $S = \kk[x_0, \dots, x_n]$ is positively $A$-graded. The $A$-grading on $S$ need not be positive for the results in Sections \ref{subsec:bgg} and \ref{unitcounit}, but this is necessary for Sections \ref{EFS} - \ref{LSTsection}. 

Let $E = \Lambda_{\kk}(e_0, \dots, e_n)$ be an exterior algebra, equipped with the $A \oplus \Z$-grading given by $\deg(e_i) = (-\deg(x_i); -1)$. We call the $\Z$-grading on $E$ the \defi{auxiliary grading}. In this paper, a \defi{differential $E$-module} is an $A \oplus \Z$-graded $E$-module equipped with a square 0 endomorphism of degree $(0 ; -1)$. All $E$-modules are right modules, but any right $E$-module $M$ can be considered as a left $E$-module with action $em = (-1)^{\on{aux}(e)\on{aux}(m)}me$, where $\on{aux}( - )$ denotes the degree in the auxiliary grading. 
\section{Background on the multigraded BGG correspondence}\label{subsec:bgg}

Let $\DM(E)$ denote the category of differential $E$-modules, and let $\Com(S)$ denote the category of complexes of $A$-graded $S$-modules. By a result of Hawwa-Hoffman-Wang in \cite{hhw}, the classical BGG correspondence \cite{bgg} generalizes to an adjunction 
\begin{equation}
\label{adjunction}
\LL : \DM(E) \leftrightarrows \Com(S) : \RR.
\end{equation}
See \cite[Section 2.2]{tate} for a detailed discussion of this multigraded BGG correspondence. We recall here the formulas for $\LL$ and $\RR$. If $D \in \DM(E)$, the complex $\LL(D)$ has terms and differential given by:
$$
\LL(D)_j = \bigoplus_{a \in A}S(-a)  \otimes_\kk D_{(a; j)} \quad \text{ and } \quad s \otimes d \mapsto  (\sum_{i = 0}^n x_is \otimes e_id) - s \otimes \del_D(d).
$$
Let $\om_E$ denote the $E$-module $\underline{\Hom}_{\kk}(E, \kk)\cong E(-\sum_{i = 0}^n \deg(x_i); -n-1)$. Given $C \in \Com(S)$, the differential $E$-module $\RR(C) \in \DM(E)$ has underlying module 
$$
\bigoplus_{j \in \Z} \bigoplus_{a \in A} (C_j)_a \otimes_\kk \om_E(-a; -j),
$$ 
and the differential acts on the $j^{\th}$ summand $\bigoplus_{a \in A} (C_j)_a \otimes_\kk \om_E(-a; -j)$ by 
$$
c \otimes f \mapsto (-1)^j (\sum_{i = 0}^n x_ic\otimes e_if) + \del_C(d) \otimes f .
$$
We will need the following key properties of the multigraded BGG functors:

\begin{prop}[\cite{hhw} Propositions 3 and 4]
\label{HHWprop}
The functors $\LL$ and $\RR$ are exact, and, given $C \in \Com(S)$ and $D \in \DM(E)$, the maps $D \to \RR\LL(D)$ and $\LL \RR(C) \to C$ arising from the unit and counit of the adjunction \eqref{adjunction} are both quasi-isomorphisms.
\end{prop}

\section{Computing the unit and counit of the multigraded BGG adjunction}
\label{unitcounit}
We will need an explicit calculation of the unit and counit of the adjunction \eqref{adjunction}; this is the content of Lemma \ref{technical} below. Lemma \ref{technical} plays a key role in the proof of Theorem \ref{introthm}.

Let $C \in \Com(S)$ and $D \in \DM(E)$, and let $\eta\colon D \xra{\simeq} \RR\LL(D)$ and $\epsilon\colon \LL \RR(C) \xra{\simeq} C$ denote the quasi-isomorphisms arising from the unit and counit of the adjunction \eqref{adjunction}. We record the following calculations of the underlying modules of $\RR\LL(D)$ and $\LL \RR(C)$:
\begin{align*}
\RR\LL(D) &= \bigoplus_{(a,j) \in A \oplus \Z}  \bigoplus_{d \in A}  S_{d - a} \otimes_\kk D_{(a;j)} \otimes_\kk \om_E(-d; -j), \\
 \LL \RR(C)_i &=  \bigoplus_{(a,j) \in A \oplus \Z}  \bigoplus_{d \in A}S(-d)\otimes_\kk (C_j)_a \otimes_\kk (\om_E)_{(d-a; i-j)}.
\end{align*}

\begin{lemma}
\label{technical}
The unit of adjunction $\eta : D \to \RR\LL(D)$ is given by the composition
$$
D \to \bigoplus_{(a; j) \in A \oplus \Z} \underline{\Hom}_{\kk}(E(a;j), D_{(a;j)})  \to \RR \LL(D),
$$
where the first map sends $d \in D$ to the collection $g_{a, j} \in \underline{\Hom}_{\kk}(E(a;j), D_{(a;j)}) $ given by
$$
g_{a,j}(e) = \begin{cases} (-1)^{j} de, & \deg(de) = (a;j), \\ 0, & \on{else;} \end{cases},
$$
and the second map identifies $\underline{\Hom}_{\kk}(E(a;j), D_{(a;j)})$ with $D_{(a;j)} \otimes_\kk \om_E(-a ; -j)$ and then embeds  into the summands with $d = a$. The counit $\epsilon\colon \LL \RR(C) \to C$ is given, in homological degree $i$, by
$$
\LL \RR(C)_i \to \bigoplus_{a \in A}  S(-a)  \otimes_\kk(C_i)_a \to C_i,
$$
where the first map projects onto the summands with $a = d$ and $i = j$, and the second map is the $S$-module action multiplied by $(-1)^i$. In particular, $\eta$ is injective, and $\epsilon$ is surjective.
\end{lemma}

\begin{proof}
Let $X \in \DM(E)$ and $Y \in \Com(S)$. Given $a \in A$ and $j \in \Z$, let $\sigma$ denote the automorphism of $\Hom_{\kk}(X_{(a; j)},  (Y_j)_{a} )$ given by sending a map $g$ to $(-1)^jg$. The adjunction isomorphism
$$
\Hom_{\DM(E)}(\LL(X), Y) \cong \Hom_{\Com(S)}(X, \RR(Y))
$$
is given by the composition
\begin{align*}
\prod_j \Hom_S( \LL(X)_j, Y_j) &= \prod_j \Hom_S(\bigoplus_{a \in A} X_{(a; j)} \otimes_\kk S(-a)  , Y_j) \\
&\xra{\cong}  \prod_{a \in A, j \in \Z} \Hom_{\kk} ( X_{(a; j)} , (Y_j)_a) \\
&\xra{\sigma} \prod_{a \in A, j \in \Z} \Hom_{\kk} ( X_{(a; j)} , (Y_j)_a) \\
& \xra{\cong} \prod_{a \in A, j \in \Z}\Hom_E(X,  \underline{\Hom}_{\kk}(E(a; j), (Y_j)_{a}  ))\\
&= \Hom_E(X, \RR(Y)),
\\
\end{align*}
where the two unlabeled isomorphisms are given by $\Hom$-$\otimes$ adjunction. Now take $X = D$ and $Y = \LL(D)$, and apply this isomorphism to $\id_{\LL(D)}$ to compute the map $\eta$. The map $\epsilon$ is computed similarly.
\end{proof}

\section{Existence and uniqueness of maximal strongly linear subcomplexes}
\label{EFS}

Recall from Definition \ref{stronglylinear} that a complex $C$ of free $A$-graded $S$-modules is \defi{strongly linear} if there exists a choice of basis of $C$ with respect to which its differentials are matrices
whose entries are $\kk$-linear combinations of the variables.

\begin{remark}
We caution that our definition of strong linearity involves a choice of coordinates $x_0, \dots, x_n$ in $S$, and it is generally not invariant under all graded ring automorphisms of $S$. For instance, if $S=\kk[x,y]$ with degrees $1$ and $2$, then the automorphism of $S$ that sends $x$ to itself and $y$ to $y - x^2$ will not preserve strongly linear complexes. However, in our main cases of interest, $S$ is the Cox ring of a toric variety or stack $X$ associated to a fan, in which case the variables of $S$ are determined by the rays of the fan. 

We also note that, if $S$ is the Cox ring of a toric variety $X$, then the notion of strong linearity depends only on the Cox ring $S$ and not its irrelevant ideal. Thus, since there exist distinct toric varieties with the same Cox ring, strong linearity is insensitive to certain aspects of the geometry of $X$.
\end{remark}

We establish some notation. Let $M$ be a graded $S$-module and $a \in A$. The $E$-module $M_a \otimes_\kk \om_E(-a ; 0)$ is a summand of $\RR(M)$; we let 
\begin{equation}
\label{Kmod}
K_a(M) = \ker(M_a \otimes_\kk \om_E(-a ; 0) \xra{\del_{\RR(M)}} \RR(M)).
\end{equation}

The goal of this section is to prove the following result, which implies Theorem \ref{introthm}:

\begin{thm}
\label{maxlinear}
Let $M$ be an $A$-graded $S$-module that is generated in a single degree $d$, and let $F$ be the minimal free resolution of $M$. \begin{enumerate}
\item $\LL(K_d(M))$ is a strongly linear, quasi-split subcomplex of $F$ (see the introduction for the definition of a quasi-split subcomplex).
\item Any strongly linear, quasi-split subcomplex of $F$ is isomorphic to a quasi-split subcomplex of $\LL(K_d(M))$.
\end{enumerate}
Thus, $\LL(K_d(M))$ is the unique maximal strongly linear, quasi-split subcomplex of $F$. In particular, if $F$ is strongly linear, then $F \cong \LL(K_d(M))$.
\end{thm}

The proof of Part (1) uses Theorem \ref{samemods} below, which says, roughly, that the minimal free resolution of $M$ is a ``deformation" of $\LL(H(\RR(M)))$. Part (2) relies on the explicit calculation of the unit and counit of the multigraded BGG adjunction from Lemma \ref{technical}.

\begin{example}\label{ex:maximal}
Say $S=\kk[x_0,x_1,x_2]$, where $\deg(x_0) = 1 = \deg(x_1)$ and $\deg(x_2) = 2$. The minimal free resolution $F$ of $M=S/(x_0,x_1^2,x_2)$ is
\begin{footnotesize}
$$
S\xleftarrow{\left[\begin{smallmatrix} y_0&y_2&y_1^2\end{smallmatrix}\right]} S(-1)\oplus S(-2)^2  
\xleftarrow{\left[\begin{smallmatrix} y_2&0&y_1^2&\\-y_0&-y_1^2&0&\\0&y_2&-y_0\end{smallmatrix}\right]}  
S(-3)^2\oplus S(-4)  \xleftarrow{\left[\begin{smallmatrix} -y_1^2 \\y_0\\y_2\end{smallmatrix}\right]} S(-5).
$$
\end{footnotesize}
The maximal strongly linear quasi-split subcomplex of $F$ is
\[
\LL(K_{\zero} (M)) =  [S\xleftarrow{\left[\begin{smallmatrix} y_0&y_2\end{smallmatrix}\right]} S(-1)\oplus S(-2)\xleftarrow{\left[\begin{smallmatrix} y_2\\ -y_0\end{smallmatrix}\right]} S(-3)\gets 0].
 \]
\end{example}

\subsection{Proof of Theorem \ref{maxlinear}(1)}

We begin by recalling the notion of a contraction of one chain complex onto another, which was introduced by Eilenberg-Maclane \cite[\S 12]{EM}.

\begin{defn}
Given two chain complexes $C$ and $C'$ in an abelian category, a \defi{contraction of $C$ onto $C'$} is a triple $(\pi, \iota, h)$, where $\pi : C \to C'$ and $\iota : C' \to C$ are morphisms of complexes such that $\pi\iota = \id_{C'}$, $h : C \to C$ is a null homotopy of $\id_C - \iota\pi$, and the relations
$$
\pi h = 0, \quad h\iota = 0, \quad h^2 = 0
$$
are satisfied.
\end{defn}

\begin{construction}
\label{contractionex}
The contractions we consider in this paper arise in the following way. Let $D \in \DM(E)$, and let $\overline{\LL(D)}$ denote the complex with underlying $S$-module identical to $\LL(D)$ and $j^{\th}$ differential induced by $-\del_D$. Similarly, let $\overline{\LL(H(D))}$ denote the complex with underlying module identical to $\LL(H(D))$ and trivial differential. We now construct a contraction of $\overline{\LL(D)}$ onto $\overline{\LL(H(D))}$. 

The complex $\overline{\LL(D)}$ splits $S$-linearly; that is, letting $Z_j$ and $B_j$ denote the $j$-cycles and $j$-boundaries in $\overline{\LL(D)}$, we may choose an $S$-module decomposition
$$
\overline{\LL(D)}_j = B_j \oplus H_j \oplus L_j
$$
for all $j$ such that $B_j \oplus H_j = Z_j$. For each $j$, let $g_j : L_j \to B_{j -1}$ denote the isomorphism such that the $j^{\th}$ differential on $\overline{\LL(D)}$ can be expressed as
$$
\begin{pmatrix}
0 & 0 & g_j \\
0 & 0 & 0 \\
0 & 0 & 0
\end{pmatrix}.
$$
Using the isomorphism $H_j \cong H_j(\overline{\LL(D)}) \cong \LL(H(D))_j$, we define morphisms
$$
\iota : \overline{\LL(H(D))} \into \overline{\LL(D)} \quad \text{and} \quad \pi : \overline{\LL(D)} \onto \overline{\LL(H(D))}
$$
of complexes such that $\pi \iota = \id$. There is a null homotopy $h$ of $\id  - \iota  \pi$ given, in degree $j$, by 
$$
\begin{pmatrix}
0 & 0 & 0 \\
0 & 0 & 0\\
g_{j+1}^{-1} & 0 & 0
\end{pmatrix}.
$$
Since $\pi  h = 0$, $h \iota = 0$, and $h^2 = 0$, the triple $(\pi, \iota, h)$ is a contraction of $\overline{\LL(D)}$ onto $\overline{\LL(H(D))}$.  
\end{construction}

\begin{example}
Suppose $S = \kk[x]$, where $\deg(x) = 1$. Let us apply Construction \ref{contractionex} in the case where $D = \RR(S)$. We have $\RR(S) = \bigoplus_{d \ge 0} \om_E(-d ; 0)$, with differential given by multiplication by the exterior variable $e$. Let $F = \bigoplus_{d \ge 0} S(-d)$. The complex $\LL\RR(S)$ has the form
$$
0 \from F \from F(-1) \from 0,
$$ 
with differential given by the matrix
$$
\begin{pmatrix}
x & 0 & 0 & \cdots \\
-1 & x & 0 & \cdots \\
0 & -1 & x & \cdots \\
0 & 0 & -1 & \cdots \\
\vdots & \vdots & \vdots & \ddots \\
\end{pmatrix}.
$$
The complex $\overline{\LL\RR(S)}$ has the same underlying module and differential 
$$
\begin{pmatrix}
0 & 0 & 0 & \cdots \\
-1 & 0 & 0 & \cdots \\
0 & -1 & 0 & \cdots \\
0 & 0 & -1 & \cdots \\
\vdots & \vdots & \vdots & \ddots \\
\end{pmatrix}.
$$
Notice that the homology of $\overline{\LL\RR(S)}$ is $S$, concentrated in degree 0. To build a contraction of $\overline{\LL\RR(S)}$ onto $\overline{\LL(H(\RR(S)))}$, we take $B_0 = F(-1) = L_1$, $H_0 = S$, and $L_0 = B_1 = H_1 = 0$; and then follow the recipe in Construction \ref{contractionex}.

\end{example}

We now recall the notion of a minimal free resolution of a complex of $A$-graded $S$-modules. We begin with the following

\begin{defn}
\label{boundedbelow}
Let $M$ be an $A$-graded $S$-module. We say $M$ is \defi{bounded above (resp. below) in $A$-degrees} if there exists a homomorphism $\theta\colon A \to \Z$ such that
\begin{enumerate}
\item $\theta$ induces a positive $A$-grading on $S$, and 
\item We have $\bigoplus_{\theta(a) = i} M_a = 0$ for all $i \gg 0$ (resp. $i \ll 0$). That is, $M$ is bounded above (resp. below) with respect to the $\Z$-grading induced by $\theta$. 
\end{enumerate}
Similarly, given an $A \oplus \Z$-graded $E$-module $N$, we say $N$ is \defi{bounded above/below in $A$-degrees} if there exists a map $\theta$ satisfying (1) such that the $\Z$-grading given by $N_i = \bigoplus_{j \in \Z} \bigoplus_{\theta(a) = i} N_{(a;j)}$ is bounded above/below.
\end{defn}

\begin{remark}
The existence of a map $\theta$ satisfying (1) in Definition \ref{boundedbelow} is assumed, since $S$ is positively $A$-graded. Thus, any finitely generated $S$-module is bounded below in $A$-degrees.
\end{remark}

If $F$ and $F'$ are minimal complexes of free $A$-modules that are homologically bounded below and whose terms are bounded below in $A$-degrees, then any quasi-isomorphism from $F$ to $F'$ is a chain isomorphism. This follows essentially from the proof of \cite[Proposition 4.4.1]{roberts}, noting that we may apply the graded version of Nakayama's Lemma to the terms of $F$ and $F'$.

\begin{defn}
\label{exuni}
Given a complex $C$ of $A$-graded $S$-modules, a \defi{minimal free resolution of $C$} is a quasi-isomorphism $F \xra{\simeq} C$, where $F$ is a minimal complex of free $A$-graded $S$-modules that is homologically bounded below and whose terms are bounded below in $A$-degrees.
\end{defn}

\begin{remark}
It follows from the above discussion that, when a minimal free resolution of a complex $C$ of $A$-graded $S$-modules exists, it is unique up to isomorphism. 
\end{remark}

The main ingredient in our proof of Theorem \ref{maxlinear}(1) is the following result:

\begin{thm}
\label{samemods}
Let $G$ be a homologically bounded below complex of $A$-graded $S$-modules whose terms are bounded below in $A$-degrees. Using the notation of Construction \ref{contractionex}, choose a contraction $(\pi, \iota, h)$ of $\overline{\LL\RR(G)}$ onto $\overline{\LL(H(\RR(G)))}$. Let $\del'$ denote the $S$-linear endomorphism of $\overline{\LL\RR(G)}$ given by multiplication on the left by $\sum_{i = 0}^n x_i \otimes e_i$, and let $F$ be the complex with underlying module the same as $\LL(H(\RR(G)))$ and differential given by 
\begin{equation}
\label{perturbdiff}
\delta + \sum_{i \ge 2} \delta_i, \text{ where $\delta$ is the differential on $\LL(H(\RR(G)))$, and } \delta_i = (-1)^{i - 1} \pi (\del' h)^{i - 1}\del'  \iota.
\end{equation}
The complex $F$ is the minimal free resolution of $G$, in the sense of Definition \ref{exuni}. Moreover, each $\delta_i$ is an $i$-fold composition of strongly linear maps. 
\end{thm}

The proof of Theorem \ref{samemods} requires the following technical lemma, which is a consequence of the Basic Perturbation Lemma in homological algebra. 

\begin{lemma}
\label{perturbation}
Let $D \in \DM(E)$, and assume $D$ is bounded below in $A$-degrees. Using the notation of Construction \ref{contractionex}, choose a contraction $(\pi, \iota, h)$ of $\overline{\LL(D)}$ onto $\overline{\LL(H(D))}$. The complex $\LL(D)$ is homotopy equivalent to a minimal complex of $A$-graded $S$-modules whose underlying $S$-module coincides with $\LL(H(D))$ and whose differential is given by  the formula in \eqref{perturbdiff}.
\end{lemma}

\begin{remark}
\label{lemmaremark}
In the case where $S$ is standard graded, a result similar to Lemma \ref{perturbation} is proven by Eisenbud-Fl\o ystad-Schreyer \cite[Corollary 3.6]{EFS}, as an application of \cite[Lemma 3.5]{EFS}. However, outside of the standard graded case, the output of the multigraded BGG functor $\LL$ cannot be interpreted as the totalization of a bicomplex, and so \cite[Lemma 3.5]{EFS} does not apply to our situation. This is why we need homological perturbation techniques to prove Lemma \ref{perturbation}. This connection between Eisenbud-Fl\o ystad-Schreyer's results and
the Basic Perturbation Lemma goes back to Coand\u{a} in \cite{coanda}. 
\end{remark}

\begin{proof}[Proof of Lemma \ref{perturbation}]
The $S$-linear endomorphism $\del'$ of $\overline{\LL(D)}$ given by multiplication on the left by $\sum_{i = 0}^n x_i \otimes e_i$ is a ``perturbation" of the differential on $\overline{\LL(D)}$, meaning that $(\del')^2 = 0$ and that adding $\del'$ to the differential on $\overline{\LL(D)}$ yields a complex (namely $\LL(D)$, in this case). The $S$-linear endomorphism $h\del'$ of $\overline{\LL(D)}$ sends an element in the summand $D_{(a; j)} \otimes S(-a)$ of $\overline{\LL(D)}$ to $\bigoplus_{i = 0}^n D_{(a - \deg(x_i); j)} \otimes S(-a + \deg(x_i))$. 
By our bounded below assumption on $D$, the endomorphism $h \del'$ is therefore locally nilpotent, in the sense of \cite[Remark A.6]{coanda}. The statement now follows from the Basic Perturbation Lemma \cite[Lemma A.4]{coanda}.
\end{proof}

\begin{proof}[Proof of Theorem \ref{samemods}]
By Proposition \ref{HHWprop}, there is a quasi-isomorphism $\LL \RR(G) \xra{\simeq} G$. Since the underlying module of $G$ is bounded below in $A$-degrees, the same is true of the $E$-module $\RR(G)$. It therefore follows from Lemma \ref{perturbation} that the complex $\LL \RR(G)$ is homotopy equivalent to a minimal complex $F$ whose underlying $S$-module coincides with $\LL(H(\RR(G)))$ and whose differential is given by the formula \eqref{perturbdiff}. Composing, we arrive at a quasi-isomorphism $F \xra{\simeq} G$. The complex $F$ is homologically bounded below and has terms that are bounded below in $A$-degrees, since $G$ has these properties; it follows that $F$ is the minimal free resolution of $G$, in the sense of Definition \ref{exuni}. Finally, observe that, for $i \ge 1$, the endomorphism $(\del'h)^{i-1}\del'$ of $\overline{\LL\RR(G)}$ is an $i$-fold composition of strongly linear maps.
\end{proof}

\begin{remark}
A version of Theorem \ref{samemods} involving differential $E$-modules, rather than complexes of $S$-modules, can be proven in much the same way. Since we do not need such a statement in this paper, we leave the details to the reader.
\end{remark}

\begin{remark}
\label{rmks:linear part}
Let $F$ be a bounded complex of finitely generated, free $S$-modules. In the standard graded case, the complex $\LL(H(\RR(F)))$ is called the ``linear part" of the minimal free resolution of $F$ \cite[Section 3]{EFS}. In other words, if we choose a basis for $F$ and write the differential $\del_F$ as a matrix, $\LL(H(\RR(F)))$ is isomorphic to the complex obtained by erasing any term in any entry of $\del_F$ that is not a $\kk$-linear combination of the variables. Theorem~\ref{samemods} suggests that one may extend Eisenbud-Fl{\o}ystad-Schreyer's notion of the ``linear part'' of a complex to the multigraded setting in the following way. We define the \defi{strongly linear part} (not to be confused with the strongly linear \emph{strand}, which we define in the next section) of a minimal, homologically bounded below complex $F$ of free $S$-modules whose terms are bounded below in $A$-degrees to be the complex $\LL(H(\RR(F)))$. 

We will not undertake a detailed study of strongly linear parts in this paper. However, let us compute one example of a strongly linear part. Say $S = \kk[x_0,x_1,x_2,x_3]$ is the Cox ring of a Hirzebruch surface of type 3, so that $A =\ZZ^{2}$, and the grading on $S$ is given by $\deg(x_0)=(1,0) = \deg(x_2), \deg(x_1)=(-3,1)$, and 
$\deg(x_3)=(0,1).$ The minimal free resolution $F$ of $M = S/(x_3 - x_0^3x_1, x_2)$ is the Koszul complex
\begin{equation}
\label{hirzres}
0 \from S \xla{\left[\begin{smallmatrix} x_3 - x_0^3x_1 & x_2 \end{smallmatrix}\right]} S(0,-1) \oplus S(-1,0) \xla{\left[\begin{smallmatrix} -x_2 \\ x_3 - x_0^3x_1 \end{smallmatrix}\right]} S(-1,-1) \from 0.
\end{equation}
One can compute directly that $\LL(H(\RR (M)))$ is the complex
$$
0 \from S \xla{\left[\begin{smallmatrix} x_3  & x_2 \end{smallmatrix}\right]} S(0,-1)^{2} \xla{\left[\begin{smallmatrix} -x_2  \\ x_3  \end{smallmatrix}\right]} S(0,-2) \from 0.
$$
\end{remark}

Let us now turn to the proof of Theorem \ref{maxlinear}(1). We start with some notation.

\begin{notation}
\label{effective}
Let $M$ be an $A$-graded $S$-module. 
We let $\Eff(M)$ denote the set $\{a \in A \text{ : } M_a \ne 0\}$, and we equip $\Eff(M)$ with the partial order given by $a \ge a'$ if and only if $a - a' \in \Eff(S)$. To explain the notation: when $M = S$, and $S$ is the Cox ring of a smooth projective toric variety $X$, then $\Eff(S)$ is the effective cone of $X$.
\end{notation}

We record the following simple result:
\begin{lemma}
\label{L(K)}
Let $M$ be a graded $A$-graded $S$-module, and suppose $a \in \Eff(M)$ is a minimal element under the partial ordering described in Notation \ref{effective}. Let $K_a(M)$ be as defined in \eqref{Kmod}. The natural map $K_a(M) \to H(\RR(M))$ is injective. In particular, we have a quasi-split injection
$
\LL(K_a(M)) \to \LL(H(\RR(M))).
$
\end{lemma}

\begin{proof}
By the minimality of $a$, no element of $K_a(M)$ is in the image of $\del_{\RR(M)}$, i.e. the natural map $K_a(M) \to H(\RR(M))$ is injective. The last statement follows from the exactness of the functor $\LL$, which is a consequence of Proposition \ref{HHWprop}.
\end{proof}

\begin{proof}[Proof of Theorem \ref{maxlinear}(1)]
By Lemma \ref{L(K)}, we have a natural quasi-split injection
$$
\gamma\colon \LL(K_d(M)) \to \LL(H(\RR(M))).
$$
Since $S$ is positively graded, and $M$ is generated in a single degree, $M$ is bounded below in $A$-degrees; we can therefore apply Theorem \ref{samemods} to $M$. We need only show that $\delta_i \gamma= 0$ for $i \ge 2$, where the maps $\delta_i$ are as in the statement of Theorem \ref{samemods}. In fact, one can check that $h \del' \iota \gamma = 0$: the key point is that no element of $K_d(M)$ is in the image of $\del_{\RR(M)}$, and this forces $h$ to vanish on the image of $\del' \iota \gamma$. 
\end{proof}

\subsection{Proof of Theorem \ref{maxlinear}(2)}

We will need the following multigraded analogue of \cite[Proposition 2.1]{EFS}, which shows that strongly linear complexes are exactly those obtained by applying $\LL$ to an $E$-module. The proof is nearly identical to the standard graded version, so we omit it.

\begin{prop}
\label{linearprop}
A complex $C \in \Com(S)$ is strongly linear if and only if $C \cong \LL(N)$ for some $N \in \DM(E)$ with trivial differential. 
\end{prop}

\begin{remark}
There is a parallel theory of strongly linear differential $E$-modules, and a version of Proposition \ref{linearprop} holds for such objects as well.
\end{remark}
\begin{proof}[Proof of Theorem \ref{maxlinear}(2)]
Suppose $L$ is a quasi-split, strongly linear subcomplex of $F$. Applying Proposition \ref{linearprop}, choose an $E$-module $N$ such that $L \cong \LL(N)$. Let $\psi$ denote the composition
$$
N \to H(\RR\LL(N)) \to H(\RR(F)) \to H(\RR(M)),
$$
where the first map is induced by the unit $\eta$ of the adjunction \eqref{adjunction}, the second is induced by the inclusion $\LL(N) \into F$, and the third is induced by the surjection $F \onto M$. It suffices to show:
\begin{enumerate}
\item[(a)] $\psi$ is injective, and 
\item[(b)] the image of $\psi$ is contained in $K_d(M)$ (note that, since $d$ is a minimal element of $\Eff(M)$, Lemma \ref{L(K)} implies that $K_d(M)$ is a submodule of $H(\RR(M))$).
\end{enumerate}
Indeed, the result is immediate from (a), (b), and the exactness of $\LL$ (Proposition \ref{HHWprop}). 

The maps $N \to \RR\LL(N)$ and $\RR(F) \to \RR(M)$ are quasi-isomorphisms (again by Proposition \ref{HHWprop}), so, to prove (a), we need only check that the map $H(\RR\LL(N)) \to H(\RR(F))$ is injective. It follows from the naturality of the identification in \cite[Proposition 2.10(a)]{tate} that we have a commutative square
$$
\xymatrix{
\Tor^S_*(k, \LL(N)) \ar[r] \ar[d]^-{\cong} & \Tor^S_*(k, F) \ar[d]^-{\cong} \\
H(\RR\LL(N)) \ar[r] & H(\RR(F))
}
$$
of $A \oplus \Z$-graded $\kk$-vector spaces. Since the top horizontal map is injective, the bottom horizontal map is as well. This proves (a). 

As for (b): by Lemma \ref{technical}, the image of the quasi-isomorphism $\eta\colon N \xra{\simeq} \RR \LL(N)$ lies in
$$
\bigoplus_{(a, j) \in A \oplus \Z} N_{(a; j)} \otimes \om_E(-a; -j).
$$
The composition $\RR \LL(N) \into \RR(F)  \onto \RR(M)$ sends each summand $N_{(a; j)} \otimes \om_E(-a; -j)$ with $j \ne 0$ to 0. Since $M$ is generated in degree $d$, $F_0$ is generated in degree $d$ as well. It follows that $\LL(N)_0$ is generated in degree $d$, since it is a summand of $F_0$; equivalently, $N_{(a; 0)} = 0$ for $a \ne d$. We conclude that $\im(\psi) \subseteq K_d(M)$. This proves (b).

\end{proof}

\begin{remark}
Let $M$ be an $A$-graded $S$-module that is bounded below in $A$-degrees. Proposition \ref{linearprop} and \cite[Theorem 2.14]{tate} imply that the minimal free resolution of $M$ is strongly linear if and only if there is a quasi-isomorphism $H(\RR(M)) \xra{\simeq} \RR(M)$.  If we further assume that $M$ is generated in a single degree $d$, then the results in this section can be used to show that this is also equivalent to the natural map $K_d(M) \to H(\RR(M))$ being an isomorphism.
\end{remark}

\section{Linear strands of multigraded free resolutions}
\label{strandsection}
We may now define the strongly linear strand of a multigraded minimal free resolution:

\begin{defn}
\label{linstranddef}
Let $M$ be an $A$-graded $S$-module that is generated in a single degree $d$. The \defi{strongly linear strand} of the minimal free resolution $F$ of $M$ is the maximal strongly linear quasi-split subcomplex of $F$, whose existence and uniqueness is guaranteed by Theorem \ref{maxlinear}. Explicitly, the strongly linear strand of $F$ is given by $\LL(K_d(M))$, where $K_d(M)$ is as in \eqref{Kmod}.
\end{defn}

By \cite[Corollary 7.11]{eisenbud}, Definition \ref{linstranddef} recovers the usual definition of the linear strand when $S$ is standard graded.

\begin{example}
In Example~\ref{ex:maximal}, the complex $\LL(K_{\zero} (M))$ is the strongly linear strand. Additionally, a straightforward calculation shows that the strongly linear strand of the minimal free resolution \eqref{hirzres} is
$0 \from S \xla{x_2} S(-1, 0) \from 0.$ 
\end{example}

Let $M$ be an $A$-graded $S$-module that is bounded below in $A$-degrees, and let $F$ be its minimal free resolution. If $M$ is not generated in a single degree, then defining its strongly linear strand as the maximal strongly linear quasi-split subcomplex of $F$ no longer makes sense, as the following simple example in the standard grading setting illustrates. 

\begin{example}\label{ex:more than one degree}
Let $S=\kk[x]$ with $\deg(x)=1$, and let $M=S/(x) \oplus S(-1)/(x)$. The minimal free resolution $F$ of $M$ is
\[
0 \from S\oplus S(-1) \xla{\left[\begin{smallmatrix} x&0\\0&x\end{smallmatrix} \right]} S(-1) \oplus S(-2)\gets 0.
\]
The standard convention would be to say that the linear strand of $F$ is $S\overset{x}{\gets} S(-1)$. But, $F$ itself is strongly linear, so its linear strand is not the maximal strongly linear quasi-split subcomplex of $F$.
\end{example}

Instead, we use BGG to give a more general definition of strongly linear strands.

\begin{defn}
\label{moregeneral}
Let $M$ be an $A$-graded $S$-module that is bounded below in $A$-degrees, $F$ its minimal free resolution, and $a \in \Eff(M)$ (see Notation \ref{effective}) a minimal element. The \defi{$a$-strongly linear strand of $F$} is defined to be $\LL(K_a(M))$, where $K_a(M)$ is as defined in \eqref{Kmod}. When $\Eff(M)$ has a \emph{unique} minimal element $a$ (e.g. when $A = \Z$), we say $\LL(K_a(M))$ is the \defi{strongly linear strand of $F$}. 
\end{defn}

\begin{remark}
When $S$ is standard graded, 
Definition \ref{moregeneral} recovers the usual definition of the linear strand \cite[Corollary 7.11]{eisenbud}.
\end{remark}

\begin{lemma}
\label{nonzero}
Let $F$ be as in Definition \ref{moregeneral}, and let $a$ be a minimal element of $\Eff(F)$. The $a$-strongly linear strand of $F$ is a nonzero quasi-split subcomplex of $F$.
\end{lemma}

\begin{proof}
The $a$-strongly linear strand is a quasi-split subcomplex of $F$ by Lemma~\ref{L(K)}. It is nonzero because a generator of the socle of $\bigoplus_{j \in \Z} (F_j)_a \otimes_\kk\om_E(-a ; -j)$ always gives a nonzero element of $K_a(F)$. 
\end{proof}

\section{A multigraded Linear Syzygy Theorem}
\label{LSTsection}

We now prove Theorem \ref{intromlst} and Corollary~\ref{greencor}.
By relying on the theory of strongly linear strands that we have developed, we are able to largely derive these results from adaptations of arguments of Eisenbud and Green. We fix the following

\begin{notation}
\label{notation}
Let $a \in A$, and let $W \subseteq S$ be the $\kk$-vector subspace of $S$ generated by the variables. Let $R_a(M)$ denote the subvariety of $\Spec(\kk[W \otimes_\kk M_a])$ defined by:
$$
R_a(M) = \{w \otimes m \in W \otimes_\kk M_a \text{ : } wm = 0 \text{ in } M\}.
$$
\end{notation}
We will prove the following slightly more general version of Theorem \ref{intromlst} involving modules that are not necessarily generated in a single degree:

\begin{thm}
\label{LST}
Let $M$ be an $A$-graded $S$-module that is bounded below in $A$-degrees, and let $F$ be its minimal free resolution. Let $a \in \Eff(M)$ (see Notation \ref{effective}) be a minimal element. The length of the strongly linear strand of $F$ is at most $\max\{ \dim_\kk M_{a} - 1, \dim R_a(M)\}$.
\end{thm}

We need the following lemma.
\begin{lemma}
\label{indstep}
Let $M$ and $a$ be as in Theorem \ref{LST}. Let $I \subseteq \{0,\dots, n\}$ be nonempty, set $S_I := \kk[x_i \text{ : } i \in I]$, and define $E_I$ similarly. Let $M_I$ denote the $S_I$-module obtained by applying restriction of scalars to $M$ along the inclusion $S_I \into S$. 
\begin{enumerate}
\item $K_a(M_I) = \{y \in K_a(M) \text{ : } ye_i = 0 \text{ for all $i \notin I$}\}$.
\item Set $c = n + 1 - \#I$. If the length of the $a$-strongly linear strand of the minimal free resolution of $M_I$ is $\l$, then the length of the $a$-strongly linear strand of the minimal free resolution of $M$ is at most $\l + c$.
\end{enumerate}
\end{lemma}

\begin{proof}
The proof of (1) is the same as that of \cite[Corollary 7.12]{eisenbud}. The proof of (2) is essentially the same as that of \cite[Corollary 7.13]{eisenbud}, but slightly different notationally, so we include it here. By induction, we may assume $c = 1$; say $I = \{0, \dots, \widehat{i}, \dots, n\}$. By (1), we have a left exact sequence
$
0 \to K_a(M_I) \into K_a(M) \xra{e_i} K_a(M)(-\deg(x_i); -1).
$
The image of the rightmost map lies in $K_a(M_I)(-\deg(x_i);-1)$. It follows that $K_a(M)$ vanishes in auxiliary degrees $i > \l + 1$.
\end{proof}

\begin{proof}[Proof of Theorem \ref{LST}]
We proceed exactly as in \cite[\S 7D]{eisenbud}. Write $m_a = \dim_\kk M_a$. Assume first that $\dim(R_a(M)) \le m_a - 1$. We must show that the length of the $a$-strongly linear strand $L$ of $F$ is at most $m_a - 1$. Given a $\kk$-vector space $N$, we let $N^* = \Hom_{\kk}(N, \kk)$. Let
$$
Q = \coker(\bigoplus_{i = 0}^{n} M_{a + \deg(x_i)}^* \otimes_\kk E(\deg(x_i) + a; 1) \xra{\a} M_{a}^* \otimes_\kk E(a ;0)),
$$
where $\a$ is the dual of the map $M_a \otimes_\kk \om_E(-a ; 0) \to  \bigoplus_{i = 0}^n M_{a + \deg(x_i)} \otimes_\kk \om_E(-a -\deg(x_i); -1)$ induced by $\del_\RR$. By definition, $L = \LL(Q^*)$. It suffices to show that $Q$ is concentrated in auxiliary degrees $-(m_a - 1), \dots, 0$. Since $Q$ is generated in auxiliary degree $0$, it suffices to show that $Q$ is annihilated by every monomial in $E$ with $m_a$ factors. The proof in this case now follows exactly as in \cite[\S 7D]{eisenbud}. The case where $\dim(R) \ge m_0$ also follows just as in \cite[\S 7D]{eisenbud}, noting that one uses Lemma \ref{indstep}(2) for the inductive step.
\end{proof}

\begin{remark}
\label{rem:sharp}
In the standard graded case, both of the bounds $\dim R_a(M)$ and $\dim_\kk M_a - 1$ may be achieved. Indeed, the bound $\dim R_a(M)$ is attained by taking $M = \kk$ \cite[Example 7.2]{eisenbud}, and the bound $\dim_\kk M_a - 1$ is achieved by taking $M = \Ext_S^{n-1}(S/I, S(n+1))$, where $I$ is the ideal defining a rational normal curve in $\PP^n$ \cite[Example 7.3]{eisenbud}. Similarly, the bounds in Theorem~\ref{LST} are sharp in the nonstandard graded setting as well. Once again, taking $M = \kk$ gives an example of a strongly linear strand of length $\dim R_a(M)$ over any positively $A$-graded polynomial ring $S$. See Example~\ref{ex:cor} below for an example of a module $M$ over a nonstandard $\Z$-graded ring with strongly linear strand of length $\dim_\kk M_a - 1$; this is a weighted projective variant of \cite[Example 7.3]{eisenbud}.\end{remark}

We finally turn to the proof of Corollary~\ref{greencor variety}, which generalizes to toric varieties a result originally proven by Green \cite{green2} over projective space. 
\begin{proof}[Proof of Corollary~\ref{greencor variety}]
(1) is immediate, since $S$ is positively graded \cite[Example A.2]{tate}. (2) follows from essentially the same argument as in \cite[Corollary 7.4]{eisenbud}, but we give the details here. Let $m \in M_a$ and $w \in W$; recall that $W \subseteq S$ is the $\kk$-vector subspace of $S$ generated by the variables. Notice that $m \otimes w \in R_a(M)$ if and only if $m \otimes w_i \in R_a(M)$ for all homogeneous components $w_i$ of $w$. Assume $m \otimes w \in R_a(M)$ and that $w$ is homogeneous; by Theorem \ref{LST}, it suffices to show that this syzygy is trivial, i.e. either $m=0$ or $w=0$.  Suppose $m\ne 0$, and let $Q$ be a maximal ideal of $S$ such that the image $m_Q$ of $m$ in the localization $M_Q$ is nonzero. Let $w_Q$ denote the image of $w$ in $(S/P)_Q$. Notice that $M_Q$ is a submodule of a free $(S/P)_Q$-module; thus, since $w_Qm_Q = 0$, this forces $w_Q = 0$, and so $w \in P$. Since $P$ does not contain a homogeneous linear form, we conclude that $w = 0$.
\end{proof}

We also give a variant of Corollary~\ref{greencor} for toric stacks.\footnote{See \cite[\S 3.1]{tate} for our definition of a projective toric stack, as well as a discussion of the relationship between a toric stack and its corresponding toric variety.  When $X$ is smooth, there is no distinction; but as in other algebraic investigations of toric geometry, sheaves are generally better behaved on the toric stack than on the corresponding toric variety.}   As there can be vector bundles on a toric stack whose corresponding sheaves over the associated toric variety are not vector bundles (e.g. $\OO(1)$ over a weighted projective stack whose weights are not all 1), this extra level of generality might prove useful for some readers.
\begin{cor}
\label{greencor}
Let $X$ be a projective toric stack with Cox ring $S$ and $P\subseteq S$ a nondegenerate, homogeneous prime ideal with $Y \subseteq X$ the corresponding integral substack.   Let $\F$ be a vector bundle on $Y$ and $M$ a submodule of the $A = \Cl(X)$-graded $S$-module given by $\bigoplus_{d \in \Eff(S)} H^0 (Y, \F(d))$. 
\begin{enumerate}
\item $\Eff(M) = \{a \in A \text{ : } M_a \ne 0\}$ contains a minimal element with respect to the partial ordering described in Notation \ref{effective}.
\item Let $a \in \Eff(M)$ be a minimal element. The $a$-strongly linear strand (see Definition \ref{moregeneral}) of the minimal free resolution of $M$ has length at most $\dim_\kk(M_a) - 1$. 
\end{enumerate}
\end{cor}
\begin{proof}
Same as the proof of Corollary~\ref{greencor}.
\end{proof}

\begin{example}
\label{ex:cor}
Let $X$ be the weighted projective space $\P(1,1,1,2,2)$ and $C\subseteq X$ the curve defined by the $2\times 2$ minors of 
$
\footnotesize
\begin{pmatrix}
x_0&x_1&x_2^2&x_3\\
x_1&x_2&x_3&x_4
\end{pmatrix}.
$
\normalsize
One can check that $C$ is a smooth rational curve and $S/I_C$ is Cohen-Macaulay; in fact, it is closely related to some of the examples from~\cite[\S2]{Np}.  Let $M$ be the canonical module of $S/I_C$, so that $\widetilde{M}=\omega_C$.  The free resolution $F$ of $M$ has Betti table
\[
\footnotesize
\begin{matrix}
& 0 & 1 & 2 & 3\\
       1: & 3 & 4 & 1 & .\\
       2: & . & 4 & 4 & .\\
       3: & . & . & 1 & .\\
       4: & . & . & . & 1
 \end{matrix}
 \normalsize
\]
and a computation in \verb|Macaulay2| shows that the strongly linear strand of $F$ is:
\[
S(-1)^3 \xleftarrow{\left(\begin{smallmatrix}
       \vphantom{\left\{1\right\}}x_{0}&0&x_{1}&0&x_{3}&0\\
       \vphantom{\left\{1\right\}}0&x_{0}&-x_{2}&x_{1}&-x_{4}&x_{3}\\
       \vphantom{\left\{1\right\}}-x_{2}&-x_{1}&0&-x_{2}&0&-x_{4}
       \end{smallmatrix}\right)
      }
       S(-2)^4\oplus S(-3)^2 \xleftarrow{\left(\begin{smallmatrix}
       \vphantom{\left\{2\right\}}-x_{1}&0&x_{3}\\
       \vphantom{\left\{2\right\}}x_{2}&0&-x_{4}\\
       \vphantom{\left\{2\right\}}x_{0}&-x_{3}&0\\
       \vphantom{\left\{2\right\}}0&x_{4}&-x_{3}\\
       \vphantom{\left\{3\right\}}0&x_{1}&-x_{0}\\
       \vphantom{\left\{3\right\}}0&-x_{2}&x_{1}
       \end{smallmatrix}\right)}
        S(-3)^1 \oplus S(-4)^2 \gets 0.
\]
Note that this has length $2$, which is $\dim M_0-1$.
\end{example}



\bibliographystyle{amsalpha}
\bibliography{Bibliography}

\begin{bibdiv}
\begin{biblist}

\bib{baranovsky}{article}{
      author={Baranovsky, Vladimir},
       title={{BGG} correspondence for toric complete intersections},
        date={2007},
     journal={Moscow Mathematical Journal},
      volume={7},
      number={4},
       pages={581\ndash 599},
}

\bib{botbol-chardin}{article}{
      author={Botbol, Nicol\'{a}s},
      author={Chardin, Marc},
       title={Castelnuovo {M}umford regularity with respect to multigraded
  ideals},
        date={2017},
        ISSN={0021-8693},
     journal={J. Algebra},
      volume={474},
       pages={361\ndash 392},
}

\bib{tate}{article}{
      author={Brown, Michael~K.},
      author={Erman, Daniel},
       title={Tate resolutions on toric varieties},
        date={2021},
     journal={to appear in the Journal of the European Mathematical Society
  (JEMS)},
}

\bib{Np}{article}{
      author={Brown, Michael~K.},
      author={Erman, Daniel},
       title={Linear syzygies of curves in weighted projective spaces},
        date={2023},
     journal={arXiv:2301.09150},
}

\bib{positivity}{article}{
      author={Brown, Michael~K.},
      author={Erman, Daniel},
       title={Positivity and nonstandard graded {B}etti numbers},
        date={2023},
     journal={to appear in the Bulletin of the London Mathematical Society},
}

\bib{beilinson}{article}{
      author={Beilinson, Alexander~A.},
       title={Coherent sheaves on $\mathbb{P}^n$ and problems of linear
  algebra},
        date={1978},
     journal={Functional Analysis and its Applications},
      volume={12},
      number={3},
       pages={214\ndash 216},
}

\bib{BES}{article}{
      author={Berkesch, Christine},
      author={Erman, Daniel},
      author={Smith, Gregory~G.},
       title={Virtual resolutions for a product of projective spaces},
        date={2020},
     journal={Algebraic Geometry},
      volume={7},
      number={4},
       pages={460\ndash 481},
}

\bib{bgg}{article}{
      author={Bernstein, I.N.},
      author={Gel'fand, I.M.},
      author={Gel'fand, S.I.},
       title={Algebraic bundles over $\mathbb{P}^n$ and problems of linear
  algebra,},
        date={1978},
     journal={Funct Anal. and Appl},
      volume={12},
       pages={212\ndash 214},
}

\bib{bruce-heller-sayrafi}{misc}{
      author={Bruce, Juliette},
      author={Heller, Lauren~Cranton},
      author={Sayrafi, Mahrud},
       title={Characterizing multigraded regularity on products of projective
  spaces},
        date={2021},
        note={arXiv:2110.10705},
}

\bib{berkesch-klein-loper-yang}{article}{
      author={Berkesch, Christine},
      author={Klein, Patricia},
      author={Loper, Michael~C.},
      author={Yang, Jay},
       title={Combinatorial aspects of virtually {C}ohen-{M}acaulay sheaves},
        date={2021},
     journal={S\'{e}m. Lothar. Combin.},
      volume={85B},
       pages={Art. 63, 13},
}

\bib{BS22}{article}{
      author={Brown, Michael~K.},
      author={Sayrafi, Mahrud},
       title={A short resolution of the diagonal for smooth projective toric
  varieties of {P}icard rank 2},
        date={2022},
     journal={to appear in Algebra \& Number Theory},
}

\bib{CK}{article}{
      author={Canonaco, Alberto},
      author={Karp, Robert~L.},
       title={Derived autoequivalences and a weighted {B}eilinson resolution},
        date={2008},
     journal={Journal of Geometry and Physics},
      volume={58},
      number={6},
       pages={743\ndash 760},
}

\bib{coanda}{article}{
      author={Coand{\u{a}}, Iustin},
       title={The {H}orrocks correspondence for coherent sheaves on projective
  spaces},
        date={2010},
     journal={Homology, Homotopy and Applications},
      volume={12},
      number={1},
       pages={327\ndash 353},
}

\bib{cox}{article}{
      author={Cox, David~A.},
       title={The homogeneous coordinate ring of a toric variety},
        date={1995},
        ISSN={1056-3911},
     journal={J. Algebraic Geom.},
      volume={4},
      number={1},
       pages={17\ndash 50},
}

\bib{duarte-seceleanu}{article}{
      author={Duarte, Eliana},
      author={Seceleanu, Alexandra},
       title={Implicitization of tensor product surfaces via virtual projective
  resolutions},
        date={2020},
     journal={Math. Comp.},
      volume={89},
      number={326},
       pages={3023\ndash 3056},
}

\bib{EES}{article}{
      author={Eisenbud, David},
      author={Erman, Daniel},
      author={Schreyer, Frank-Olaf},
       title={Tate resolutions for products of projective spaces},
        date={2015},
     journal={Acta Mathematica Vietnamica},
      volume={40},
      number={1},
       pages={5\ndash 36},
}

\bib{EFS}{article}{
      author={Eisenbud, David},
      author={Fl{\o}ystad, Gunnar},
      author={Schreyer, Frank-Olaf},
       title={Sheaf cohomology and free resolutions over exterior algebras},
        date={2003},
     journal={Transactions of the American Mathematical Society},
      volume={355},
      number={11},
       pages={4397\ndash 4426},
}

\bib{eisenbud}{book}{
      author={Eisenbud, David},
       title={The geometry of syzygies: a second course in algebraic geometry
  and commutative algebra},
   publisher={Springer Science \& Business Media},
        date={2005},
      volume={229},
}

\bib{ek}{article}{
      author={Eisenbud, David},
      author={Koh, Jee},
       title={Some linear syzygy conjectures},
        date={1991},
     journal={Adv. Math.},
      volume={90},
      number={1},
       pages={47\ndash 76},
}

\bib{EM}{article}{
      author={Eilenberg, Samuel},
      author={MacLane, Saunders},
       title={On the groups {$H(\Pi,n)$}. {I}},
        date={1953},
        ISSN={0003-486X},
     journal={Ann. of Math.},
      volume={58},
       pages={55\ndash 106},
}

\bib{EPS}{article}{
      author={Eisenbud, David},
      author={Popescu, Sorin},
      author={Schreyer, Frank-Olaf},
      author={Walter, Charles},
       title={Exterior algebra methods for the minimal resolution conjecture},
        date={2002},
     journal={Duke Mathematical Journal},
      volume={112},
      number={2},
       pages={379\ndash 395},
}

\bib{GL}{article}{
      author={Green, Mark},
      author={Lazarsfeld, Robert},
       title={Some results on the syzygies of finite sets and algebraic
  curves},
        date={1988},
     journal={Compositio Mathematica},
      volume={67},
      number={3},
       pages={301\ndash 314},
}

\bib{gllm}{article}{
      author={Gao, Jiyang},
      author={Li, Yutong},
      author={Loper, Michael~C.},
      author={Mattoo, Amal},
       title={Virtual complete intersections in $\mathbb{P}^1 \times \mathbb
  {P}^1$},
        date={2021},
     journal={J. Pure Appl. Algebra},
      volume={225},
      number={1},
}

\bib{green2}{article}{
      author={Green, Mark},
       title={Koszul cohomology and the geometry of projective varieties},
        date={1984},
     journal={Journal of Differential Geometry},
      volume={19},
      number={1},
       pages={125\ndash 171},
}

\bib{green}{article}{
      author={Green, Mark},
       title={The {E}isenbud-{K}oh-{S}tillman conjecture on linear syzygies},
        date={1999},
     journal={Inventiones Mathematicae},
      volume={136},
       pages={411\ndash 418},
}

\bib{hhw}{article}{
      author={Hawwa, F.~T.},
      author={Hoffman, J.~William},
      author={Wang, Haohao},
       title={Koszul duality for multigraded algebras},
        date={2012},
     journal={Eur. J. Pure Appl. Math.},
      volume={5},
      number={4},
       pages={511\ndash 539},
}

\bib{hnv}{article}{
      author={Harada, Megumi},
      author={Nowroozi, Maryam},
      author={Van~Tuyl, Adam},
       title={Virtual resolutions of points in {$\Bbb P^1\times\Bbb P^1$}},
        date={2022},
        ISSN={0022-4049},
     journal={J. Pure Appl. Algebra},
      volume={226},
      number={12},
       pages={Paper No. 107140, 18},
}

\bib{loper}{article}{
      author={Loper, Michael~C.},
       title={What makes a complex a virtual resolution?},
        date={2021},
     journal={Trans. Amer. Math. Soc. Ser. B},
      volume={8},
       pages={885\ndash 898},
}

\bib{M2}{misc}{
       title={Macaulay2 – a system for computation in algebraic geometry and
  commutative algebra programmed by {D}. {G}rayson and {M}. {S}tillman},
         url={http://www.math.uiuc.edu/Macaulay2/},
}

\bib{mccullough}{article}{
      author={McCullough, Jason},
       title={Stillman's question for exterior algebras and {H}erzog's
  conjecture on {B}etti numbers of syzygy modules},
        date={2019},
     journal={Journal of Pure and Applied Algebra},
      volume={223},
      number={2},
       pages={634\ndash 640},
}

\bib{MS}{book}{
      author={Miller, Ezra},
      author={Sturmfels, Bernd},
       title={Combinatorial commutative algebra},
   publisher={Springer Science \& Business Media},
        date={2004},
      volume={227},
}

\bib{MS2}{article}{
      author={Maclagan, Diane},
      author={Smith, Gregory~G.},
       title={Uniform bounds on multigraded regularity},
        date={2005},
        ISSN={1056-3911},
     journal={J. Algebraic Geom.},
      volume={14},
      number={1},
       pages={137\ndash 164},
  url={https://doi-org.ezproxy.library.wisc.edu/10.1090/S1056-3911-04-00385-6},
}

\bib{roberts}{book}{
      author={Roberts, Paul~C.},
       title={Multiplicities and {C}hern classes in local algebra},
   publisher={Cambridge University Press},
        date={1998},
      number={133},
}

\bib{RW}{article}{
      author={Raicu, Claudiu},
      author={Weyman, Jerzy},
       title={Syzygies of determinantal thickenings and representations of the
  general linear {L}ie superalgebra},
        date={2019},
     journal={Acta Mathematica Vietnamica},
      volume={44},
      number={1},
       pages={269\ndash 284},
}

\bib{SS}{article}{
      author={Schenck, Henry},
      author={Suciu, Alexander},
       title={Resonance, linear syzygies, {C}hen groups, and the
  {B}ernstein-{G}elfand-{G}elfand correspondence},
        date={2006},
     journal={Transactions of the American Mathematical Society},
      volume={358},
      number={5},
       pages={2269\ndash 2289},
}

\bib{yang}{article}{
      author={Yang, Jay},
       title={Virtual resolutions of monomial ideals on toric varieties},
        date={2021},
     journal={Proc. Amer. Math. Soc. Ser. B},
      volume={8},
       pages={100\ndash 111},
}

\end{biblist}
\end{bibdiv}

\end{document}